\documentclass{amsart}
\usepackage[utf8]{inputenc}
\usepackage[euler-digits,euler-hat-accent]{eulervm}
\usepackage{newpxtext}

\usepackage{amsmath}
\usepackage{amssymb}
\usepackage{amsthm, thmtools}
\usepackage{xfrac}
\usepackage[colorlinks=false]{hyperref}
\usepackage{nameref, cleveref}
\crefname{figure}{Fig.}{Figs.}
\usepackage[numbers, sort&compress]{natbib}
\usepackage[a4paper, left=2.5cm, right=2.5cm, top=2.5cm, bottom=2.5cm]{geometry}
\usepackage{stmaryrd}
\usepackage{pstool}
\usepackage{caption}
\usepackage{listings}
\usepackage{float}
\usepackage{placeins}
\usepackage{tikz}
\usetikzlibrary{automata, positioning, arrows}
\usetikzlibrary{arrows.meta}
\usepackage{graphicx}
\usepackage{caption}
\usepackage{subcaption}

\usepackage{todonotes}
\usepackage{amsaddr}

\newcommand{\koppa}{\text{$\shortmid$\kern-0.3605em{\raise0.341em\hbox{$\circ$}}}}

\declaretheorem[name=Theorem, refname={Theorem,Theorems}, Refname={Theorem, Theorems}]{thm}
\declaretheorem[name=Lemma, refname={Lemma, Lemmas}, Refname={Lemma, Lemmas}, sibling=thm]{lem}

\title{A short note on graphs with long Thomason chains}
\author{Marcin Briański, Adam Szady}
\address{
\textup{
$\{\texttt{marbri},\texttt{adsz}\}\texttt{@beit.tech}$
}
\\[8pt]
Beit.tech}

\date{June 13, 2020}

\hypersetup{
    pdftitle={A short note on graphs with long Thomason chains},
    pdfauthor={Marcin Briański, Adam Szady}
}
\DeclareMathOperator{\Prefx}{Pref}
\usepackage{mathrsfs}
\let\base\textsc

\def\gadget#1{\textup{\texttt{#1}}}

\newlength\baz
    \usepackage[percent]{overpic}
          
\newcommand{\podpisik}[1]{\hbox to 0.3\baz{\hfill$\gadget{#1}$\hfill}}
\newcommand{\podpisikpac}[0]{\hbox to 0.3\baz{\hfill$\gadget{\$}$\hspace{23pt}\hfill}}

\newcommand{\kawalek}[2]{
    \begin{subfigure}{.22\textwidth}
        \centering
        \includegraphics[page=#1,scale=0.4]{graphs.pdf}
        \caption*{#2}
        \vspace{.6em}
    \end{subfigure}
}

\begin{document}

\maketitle

    \begin{abstract}
    We present a family of 3-connected cubic planar Hamiltonian graphs with an exponential number of 
    steps required by Thomason's algorithm. The base of the exponent is approximately $1.1812...$, which exceeds previous results in the area.
    \end{abstract}
    
    \section{Introduction}
    
    In \cite{thomason}
    Thomason introduced a simple constructive proof of Smith's theorem.
    His algorithm, given a Hamiltonian cycle and one of its edges, finds a second Hamiltonian cycle that also contains this edge. The algorithm consists of steps called \emph{lollipops}, which alter a Hamiltonian path in a reversible, deterministic way.
    
    It was interesting for many researchers to determine whether the number of steps involved in 
    Thomason's algorithm grows polynomially with the size of the graph.
    The first one to provide a family of graphs with exponential growth of the number of steps was
    Krawczyk, see \cite{krawczyk}. More precisely, the number of steps was $\Theta(2^{n/8})$, 
    where $n$ is the number of vertices of the input graph.
    The construction by Krawczyk contained small mistakes that were subsequently corrected by Cameron. The construction was also generalised by Cameron, however number of steps required remained the same, see~\cite{cameron}.
    
    More recently published work~\cite{zhong} makes an argument for discussing cubic cyclically 4-edge connected graphs and presents a family of such graphs for which the number of steps of Thomason's algorithm grows like $\Theta(2^{n / 16})$. Unlike previously mentioned graph families, where graphs have exactly three Hamiltonian cycles regardless of the number of vertices, the number of Hamiltonian cycles in Zhong's family of graphs grows exponentially with the number of vertices. 

    Our interest in fast growing number of steps for small graphs stems from an attempt to compare scaling behaviour of the quantum equivalent of Thomason's algorithm, where both available quantum hardware and simulation capabilities limit studied graph size.
    
    \section{Description of Thomason's algorithm}
    
    For the sake of completeness, we include a brief description of Thomason's algorithm. 
    
    \begin{thm}[\cite{thomason}]\label{thm:thomason}
        Let \(G\) be a cubic graph, \(C\) be a Hamiltonian cycle in \(G\) and \(e\) be an edge of the cycle \(C\). Then the number of Hamiltonian cycles in \(G\) that contain \(e\) is even. Moreover, the proof provides an algorithm, such that given \(G\), \(C\) and \(e\), as in theorem, it outputs a Hamiltonian cycle in \(G\) different from \(C\) that also contains \(e\).
    \end{thm}
    \begin{proof}
    
        We will define an auxiliary graph \(\koppa(G)\) with vertex set consisting of all oriented Hamiltonian paths in \(G\). Consider a Hamiltonian path \(P\) in \(G\) and write \(v_1,v_2, \dots, v_n\) for the vertices of \(G\) in order that \(P\) visits them.
        As degree of \(v_n\) is 3 in \(G\), it must have a neighbour \(v_i\) with \(1 < i < n - 1\) (there is either one or two such neighbours, depending on whether or not \(P\) is a Hamiltonian cycle). Thus \(v_1,\dots, v_i, v_n, v_{n-1}, v_{n-2}, \dots, v_{i+1}\) is again an oriented Hamiltonian path and call it \(Q\).
        In this case the paths \(P\) and \(Q\) are adjacent in \(\koppa(G)\). Observe that this relation is symmetric, thus \(\koppa(G)\) is an undirected graph. See \cref{fig:lolipop} for a graphical representation of adjacent paths in \(\koppa(G)\).
        
        Clearly, the degree of any vertex in \(\koppa(G)\) is either 1 or 2, thus 
        it is a disjoint union of paths and cycles. Moreover vertices of degree 1 in \(\koppa(G)\) 
        are precisely Hamiltonian cycles in \(G\). Pairing Hamiltonian cycles if they lie in the same component of  \( \koppa(G)\) yields the desired result after observing that the first edge is always preserved between neighbours in \(\koppa(G)\).
    
    \begin{figure}
        \centering
        \vspace{10pt}
        \begin{subfigure}{\textwidth}
            \centering
            \begin{overpic}[scale=0.1]{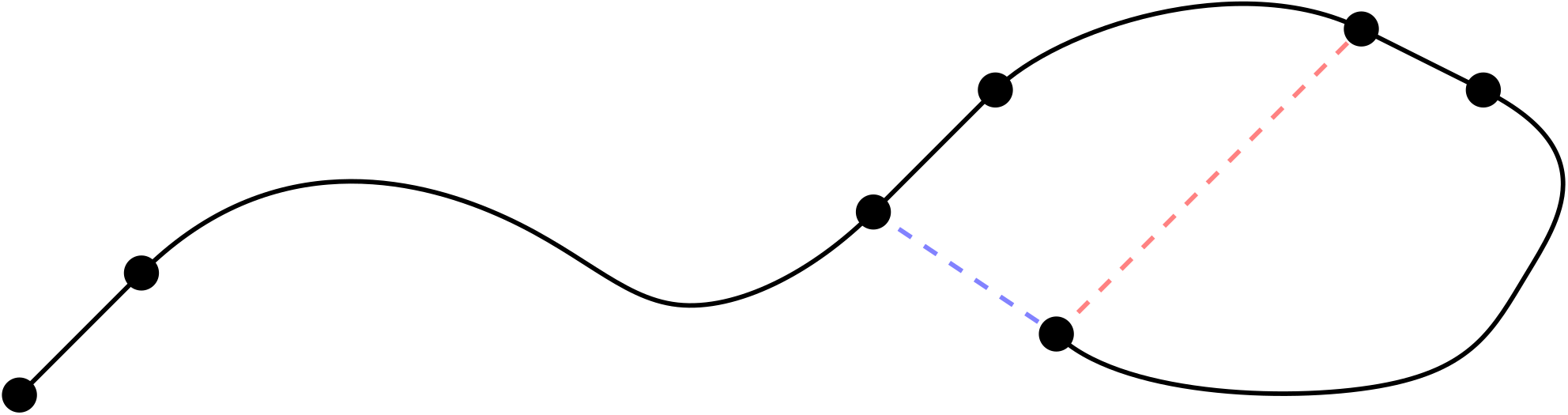}
                \put (-3,5) {\(v_1\)}
                \put (3.5, 11.5) {\(v_2\)}
                \put (61, 1) {\(v_n\)}
                \put (50, 15) {\(v_i\)}
                \put (57, 25) {\(v_{i+1}\)}
                \put (88, 27) {\(v_j\)}
                \put (97, 22) {\(v_{j+1}\)}
            \end{overpic}
            \caption*{Initial Hamiltonian path \(P\).}
            \vspace{10pt}
        \end{subfigure}
        \begin{subfigure}{0.45\textwidth}
            \centering
            \includegraphics[scale=0.1]{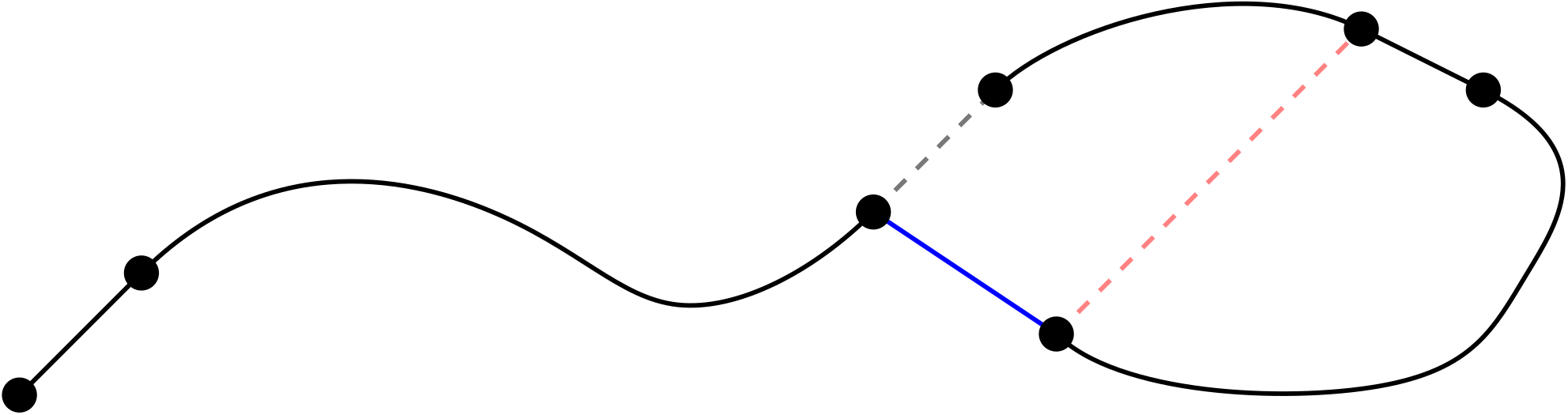}
            \caption*{Neighbour \(Q\) of \(P\) in \(\koppa(G)\).}
            \hspace{10pt}
        \end{subfigure}
        \begin{subfigure}{0.45\textwidth}
            \centering
            \includegraphics[scale=0.1]{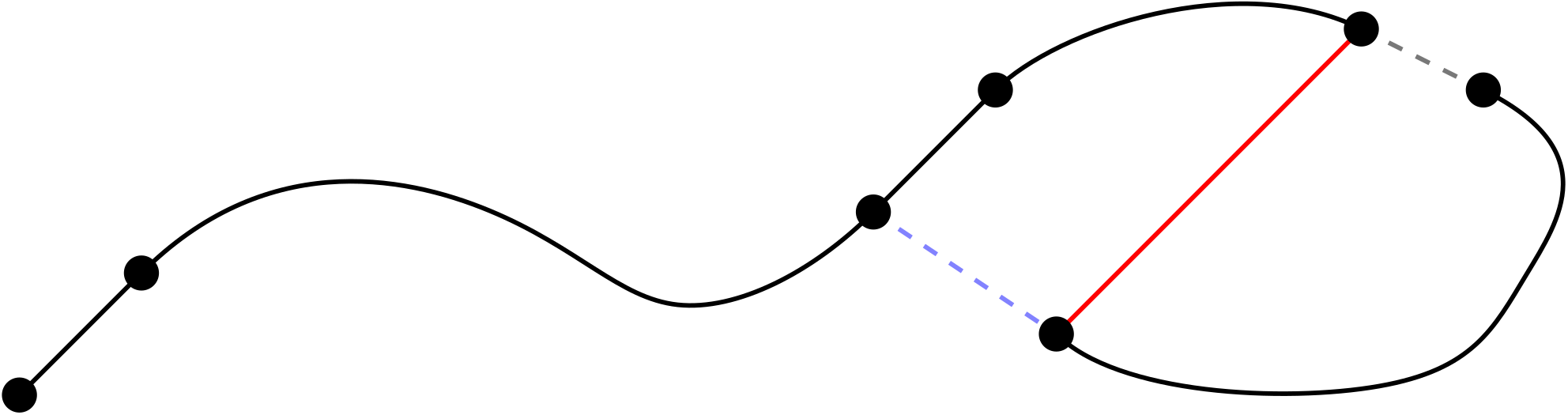}
            \caption*{Another neighbour of \(P\) in \(\koppa(G)\) (which exists whenever \(P\) is not a Hamiltonian cycle).}
            \hspace{10pt}
        \end{subfigure}
        \caption{Two ways one can lollipop a single path that is not a cycle.}
        \label{fig:lolipop}
    \end{figure}
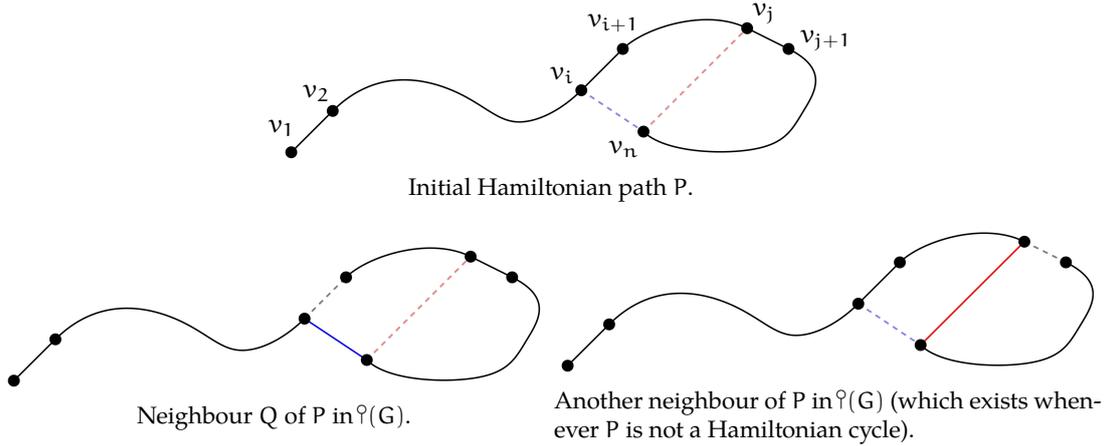
    \end{proof}
    
    The operation of changing currently considered Hamiltonian path to one adjacent in \(\koppa(G)\) is called \emph{lollipopping}.
    
    The proof of \cref{thm:thomason} is algorithmic in nature -- repeatedly lollipopping a given Hamiltonian cycle with a distinguished edge (remembering previous path in order not to move back) will eventually yield a second cycle which also contains the distinguished edge. 
    
    \section{Description of the graph family}
    The construction uses three components depicted in \Cref{fig1}: a \emph{cap} -- $K_3$ with each vertex having an additional edge connected to a vertex in the next part of the graph; a cap flipped horizontally -- a \emph{pac}; and a \emph{gadget} consisting of two vertices and edges to adjacent components.
    
    The family of graphs is indexed by natural numbers.
    The $n\text{-th}$ graph -- $G_n$ starts with a cap, followed by $n$ gadgets and finally terminates with a pac.
    Each gadget introduces 2 new vertices, so $|V(G_n)| = 2n + 6$ and $|E(G_n)| = 3n + 9$.
    Each $G_n$ is cubic, $3\text{-connected}$, planar, and has exactly three Hamiltonian cycles.
    \begin{figure}[ht]
        \centering
        \begin{subfigure}{.3\textwidth}
            \centering
            \includegraphics[page=2,scale=0.4]{graphs.pdf}
            \caption*{cap}
            \label{fig1:cap}
        \end{subfigure}
        \begin{subfigure}{.3\textwidth}
            \centering
            \includegraphics[page=1,scale=0.4]{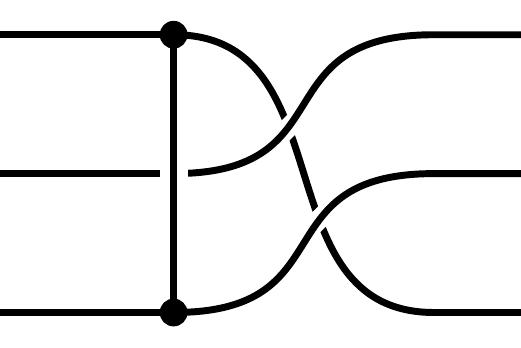}
            \caption*{gadget}
            \label{fig1:gadget}
        \end{subfigure}
        \begin{subfigure}{.3\textwidth}
            \centering
            \includegraphics[page=3,scale=0.4]{graphs.pdf}
            \caption*{pac}
            \label{fig1:pac}
        \end{subfigure}
        \caption{Graph construction components.}
        \label{fig1}
    \end{figure}
    
    \begin{figure}[ht]
        \centering
        \captionsetup{justification=centering}
        \includegraphics[page=4,scale=0.5]{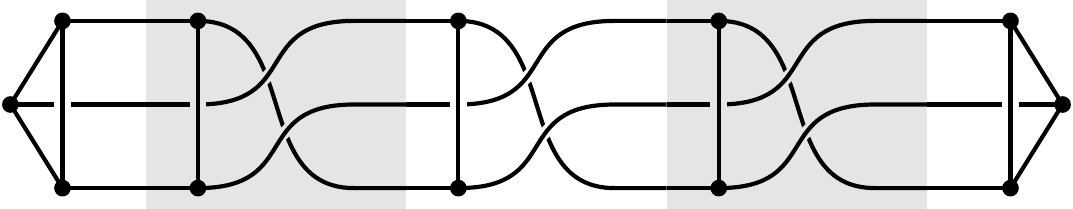}
        \caption{
            The graph $G_3$.\\
            For convenience, consecutive gadgets are presented on alternating backgrounds.
        }
        \label{fig2}
    \end{figure}
    
    \begin{figure}[ht]
        \centering
        \begin{subfigure}{\textwidth}
            \centering
            \begin{overpic}[page=31,scale=0.4]{graphs.pdf}
             \put (0,13) {$\Lambda$}
            \end{overpic}
            \caption*{$C_0$ -- the initial cycle. We will be looking for another cycle that also contains the \emph{green} edge. Which implies, in the case of this graph family, not using the edge marked red. We distinguish one special vertex of the cap -- $\Lambda$.}
        \end{subfigure}
        \\[12pt]
        \begin{subfigure}{\textwidth}
            \centering
            \begin{overpic}[page=32,scale=0.4]{graphs.pdf}
             \put (0,13) {$\Lambda$}
            \end{overpic}
            \caption*{$C_1$ -- the final cycle. Edges used in the pac are uniquely determined by the parity of $n$.}
        \end{subfigure}
        \caption{Two (out of three) Hamiltonian cycles in graph $G_n$.}
        \label{fig:p0p1}
    \end{figure}
    
    \section{Main result}
    
    The goal of this paper is to establish the following theorem.
    
    \begin{thm}\label{thm:main}
        Let $n \in \mathbb{N}$, $n \ge 3$, and consider the graph $G_{n}$. Let $C_0$ and $C_1$ be the Hamiltonian cycles in $G_n$ shown in \cref{fig:p0p1}.
        The Thomason's algorithm in $G_n$ starting with $C_0$ and green edge in \cref{fig:p0p1}, terminates with $C_1$ and takes $\Theta(c^{n})$ steps, where $ c>1 $ is some constant ($c \approx 1.3953...$).
    \end{thm}
    Taking the square root of the constant $c$ from \cref{thm:main} (as the number of vertices in $G_n$ grows like $2n$), we get the base of the exponent from the abstract.
    
    In the proof of \cref{thm:main} we will be considering paths arising during Thomason's algorithm.
    To analyse them efficiently we first introduce notation to describe how a Hamiltonian path may pass through a single gadget. We focus only on the paths starting at $\Lambda$ and using the green edge.
    Given these constraints, we consider two categories of patterns: \emph{letter patterns} -- with an end of the path on each side of the gadget, and \emph{number patterns} -- with an end of the path inside the gadget. In the case with both ends on the left side, the choice of two edges used by the path uniquely determines the pattern and this behaviour propagates all the way to the pac.
    
    Letter patterns are presented by a list of cases in \cref{fig:letter_states}.
    These are all possible ways a Hamiltonian path can start in the cap, pass through the gadget, and end in the right part of the graph.
    One can verify that this list is complete in the following way: 
    the path can either use the edge inside the gadget or not; then choosing where the path first enters the gadget and where it last leaves the gadget, uniquely determines the way the path traverses the gadget, out of which 6 are Hamiltonian, giving in total 12 patterns.
    The pattern $\gadget{Y}$ cannot follow neither any other pattern nor a cap (see \cref{fig:aandaa}), so it never arises during the algorithm and is henceforth disregarded.
    All number patterns are covered by \cref{fig:number_states}.
    The edge inside the gadget can be used -- in this case specifying the endpoint completely determines how the path must pass through the gadget. If the edge is unused, choosing how the path returns from the left side of the graph forces uniquely the path.
    
    \begin{figure}[b]
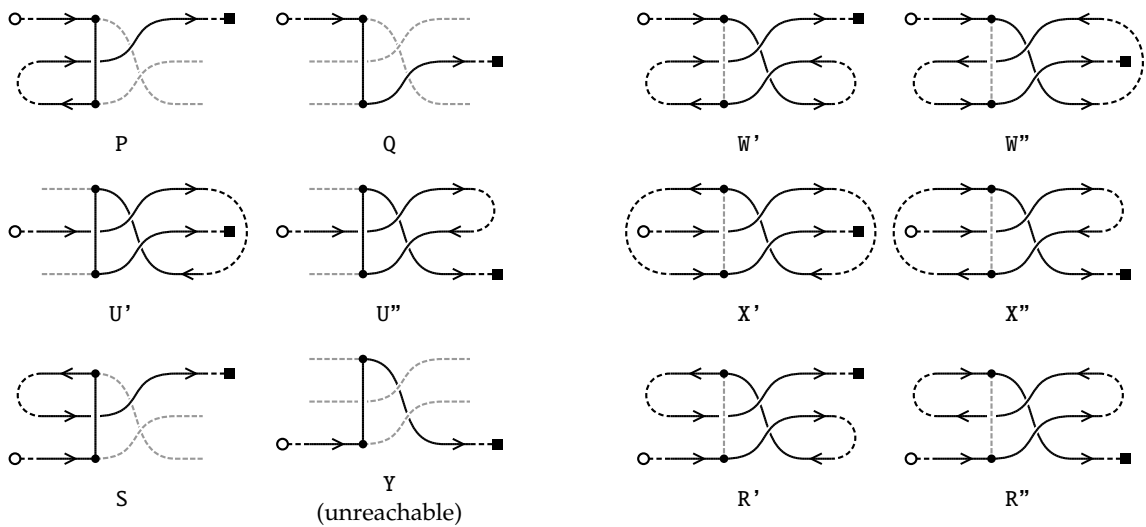

        \begin{center}
        \kawalek{5}{$\gadget{P}$} \hspace{-10pt}
        \kawalek{6}{$\gadget{Q}$} \hspace{25pt}
        \kawalek{13}{$\gadget{W'}$} \hspace{-10pt}
        \kawalek{9}{$\gadget{W''}$}
        \kawalek{8}{$\gadget{U'}$} \hspace{-10pt}
        \kawalek{7}{$\gadget{U''}$} \hspace{25pt}
        \kawalek{14}{$\gadget{X'}$} \hspace{-10pt}
        \kawalek{12}{$\gadget{X''}$}
        \kawalek{11}{$\gadget{S}$} \hspace{-10pt}
        \kawalek{16}{$\gadget{Y}$\\(unreachable)} \hspace{25pt}
        \kawalek{15}{$\gadget{R'}$} \hspace{-10pt}
        \kawalek{10}{$\gadget{R''}$}
        \end{center}
        \vspace{-1.3em}
        \caption{All possible \emph{letter patterns}. The empty circle and the black square denote respectively the beginning and the end of the path.}
        \label{fig:letter_states}
    \end{figure}

    \begin{figure}[ht!]
        \centering
        \begin{subfigure}{.3\textwidth}
            \centering
            \includegraphics[page=40,scale=0.4]{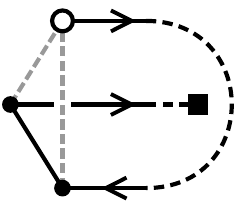}
            \label{fig:aandaa:a}
        \end{subfigure}
        \begin{subfigure}{.3\textwidth}
            \centering
            \includegraphics[page=41,scale=0.4]{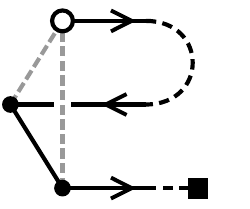}
            \label{fig:aandaa:aa}
        \end{subfigure}
        \caption{All (two) possible ways a Hamiltonian path starting at $\Lambda$, using the green edge and not ending in the cap, can pass through the cap.}
        \label{fig:aandaa}
    \end{figure}
    
    \begin{figure}[ht!]
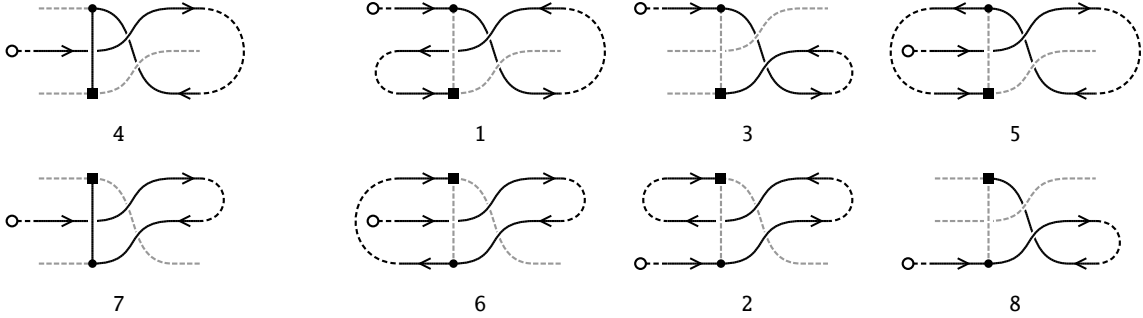

        \begin{center}
            \kawalek{20}{$\gadget4$} \hspace{25pt}
            \kawalek{17}{$\gadget1$} \hspace{-10pt}
            \kawalek{19}{$\gadget3$} \hspace{-10pt}
            \kawalek{21}{$\gadget5$}
            \kawalek{23}{$\gadget7$} \hspace{25pt}
            \kawalek{22}{$\gadget6$} \hspace{-10pt}
            \kawalek{18}{$\gadget2$} \hspace{-10pt}
            \kawalek{24}{$\gadget8$}
        \end{center}
        \vspace{-1.3em}
        \caption{
        All possible \emph{number patterns}.
        }
        \label{fig:number_states}
    \end{figure}
    
    We will call a Hamiltonian path in $G_n$ starting at $\Lambda$ a \emph{rightmost path}, if its other end is a vertex of the pac. These paths are central to our analysis of the algorithm's behaviour on the graph \(G_n\).
    
    Consider a Hamiltonian path in $G_n$ that begins in the vertex $\Lambda$.
    We assign one of the letter patterns to each of the gadgets to the left of the path’s endvertex.
    Observe that the way our path passes through the gadgets, up to the one containing endvertex, is uniquely encoded by this word. When describing a rightmost path, each such word encodes either one or two Hamiltonian paths, depending whether it ends on $\gadget{P}$, $\gadget{Q}$ or $\gadget{S}$ (in which case, there are two), or not (and there is only one).
    To understand better why this happens, note that distinction between pairs \gadget{U'} -- \gadget{U''},
    \gadget{W'} -- \gadget{W''},
    \gadget{X'} -- \gadget{X''},
    and \gadget{R'} -- \gadget{R''}
    is dependent only on the edges used in the part of the graph to the right of the considered gadget. This is because gadgets in one pair are exactly equal when considered as edge sets.
    So in the case of a rightmost path, they differ only in how the path traverses the pac, and there is a lollipop operation that maps between two possibilities. 
    Thus we map
    $\gadget{U'}, \gadget{U''} \mapsto \gadget{U}$,\quad 
    $\gadget{W'}, \gadget{W''} \mapsto \gadget{W}$,\quad
    $\gadget{X'}, \gadget{X''} \mapsto \gadget{X}$,\quad 
    and $\gadget{R'}, \gadget{R''} \mapsto \gadget{R}$.
    These labels will be used through the paper.
    
    We now introduce three lemmas, which we will use later in the proof of \cref{thm:main}.
    
    \begin{lem}[Counter Initialisation Lemma]\label{lem:init}
        Consider the two cycles $C_0$ and $C_1$ (shown in the~\cref{fig:p0p1}). Then Thomason's algorithm starting with $C_0$ terminates with $C_{1}$.
        Moreover, the first rightmost path encountered during such algorithm's run is a path described by a prefix of the string $\gadget{PQU}\hspace{2pt}\gadget{PQU}\hspace{2pt}\gadget{PQU}$...\hspace{0.7pt}, and the last rightmost path (before reaching $C_1$) is a path described by a prefix of the string $\gadget{WSQU}\hspace{2pt}\gadget{WSQU}\hspace{2pt}\gadget{WSQU}...$\hspace{0.7pt}.
        
    \end{lem}
    \begin{proof}
        By the \cref{thm:thomason}, repeatedly applying the lollipop operation will lead us to another cycle containing the green edge.
        Since the only cycle other than $C_0$ satisfying that condition is $C_1$, we are done with the first part.
        
        We prove the second assertion by applying by hand a couple of lollipops starting from $C_0$. The remaining cases follow a similar recursive pattern.
        Lollipopping the $C_0$ leads to a rightmost path made of repeated sequence of patterns: $\gadget{P}, \gadget{Q}, \gadget{U}$. For illustration, in \cref{fig:init}, the first few lollipops are shown. 
        The description of the last rightmost path can be obtained in a similar manner.
        Note that this procedure works regardless of $n$, however it is possible for the last group of symbols (i.e. either \(\gadget{PQU}\) or \(\gadget{WSQU}\)) to be only a proper prefix of \(\gadget{PQU}\) or \(\gadget{WSQU}\). \qedhere
        
    \end{proof}
    \begin{figure}[t]
        \centering
        \begin{subfigure}{\textwidth}
            \centering
            \includegraphics[page=37,scale=0.3]{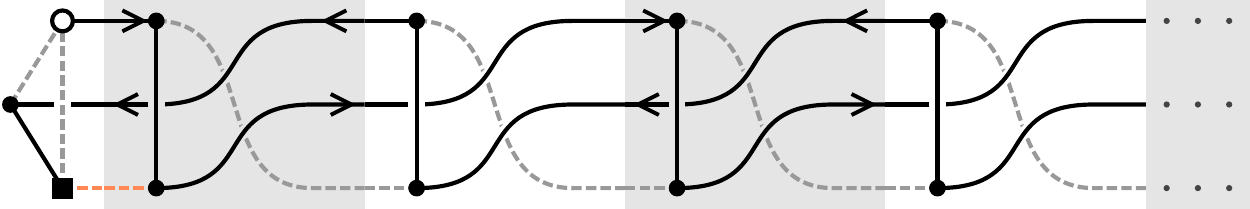}
            \caption{Initial situation (cycle $C_0$ with one edge removed).\\[10pt]}
            \label{fig:init:init}
        \end{subfigure}
        \begin{subfigure}{\textwidth}
            \centering
            \includegraphics[page=38,scale=0.3]{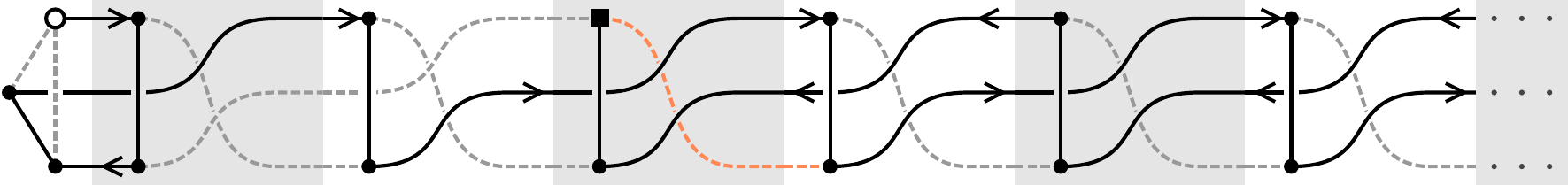}
            \caption*{\podpisik{P}\podpisik{Q}\podpisik{7}\podpisik{}\podpisik{}\podpisik{}}
            \vspace{-4pt}
            \caption{First lollipop applied.\\[10pt]}
            \label{fig:init:1}
        \end{subfigure}
        \begin{subfigure}{\textwidth}
            \centering
            \includegraphics[page=39,scale=0.3]{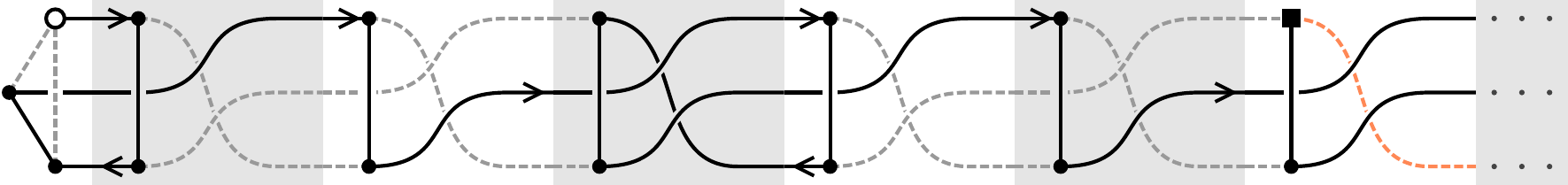}
            \caption*{\podpisik{P}\podpisik{Q}\podpisik{U}\podpisik{P}\podpisik{Q}\podpisik{7}}
            \vspace{-4pt}
            \caption{Second lollipop applied. The sequence \(\gadget{PQ7}\) has appeared again, so the pattern must repeat.}
            \label{fig:init:2}
        \end{subfigure}
        \caption{Initialisation of the first rightmost path starting from initial cycle (\subref{fig:init:init}). The next two steps, shown on \subref{fig:init:1} and \subref{fig:init:2}, prove the recursive formula for the first rightmost path's word.}
        \label{fig:init}
    \end{figure}
    
    We say that a number pattern $j$ is a \emph{bouncing pattern} if among both ways to lollipop it, the end of the Hamiltonian path after lollipopping is always to the right of the gadget. Likewise, $j$ is a \emph{conducting pattern} if the end of the Hamiltonian path after lollipopping ends on either side of the gadget, for the two ways one can lollipop a given Hamiltonian path.
    
    \begin{lem}[Bouncing Lemma] \label{lem:bounce}
        Among the number patterns (presented in \cref{fig:number_states}), $\gadget1$ and $\gadget2$ are bouncing patterns. Number patterns other than $\gadget1$ and $\gadget2$ are conducting patterns.
    \end{lem}
    \begin{proof}
    
        Again, the proof requires us to consider both ways one may lollipop the number patterns. This is easily done by hand. For illustrative purposes, we show the behaviour around $\gadget2$ and $\gadget4$ on \cref{fig:bouncing2,fig:bouncing4}, respectively. The reader is encouraged to verify the remaining cases. \qedhere
        
    \end{proof}
    \begin{figure}
        \centering
        \begin{subfigure}{.3\textwidth}
            \centering
            \includegraphics[page=25,scale=0.4]{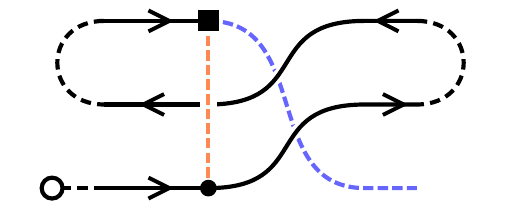}
            \caption{Pattern $\gadget{2}$}
            \label{fig:bouncing2:initial}
        \end{subfigure}
        \begin{subfigure}{.3\textwidth}
            \centering
            \includegraphics[page=26,scale=0.4]{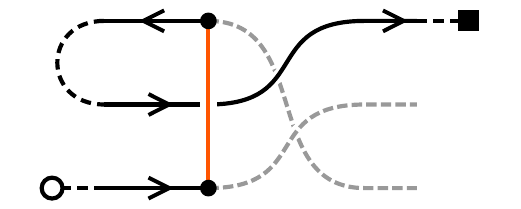}
            \caption{Pattern $\gadget{S}$}
            \label{fig:bouncing2:red}
        \end{subfigure}
        \begin{subfigure}{.3\textwidth}
            \centering
            \includegraphics[page=27,scale=0.4]{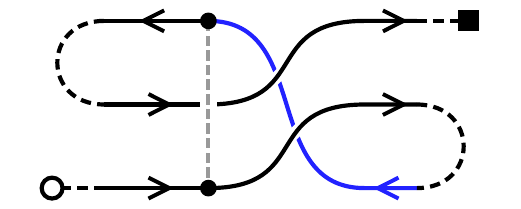}
            \caption{Pattern $\gadget{X}$}
            \label{fig:bouncing2:green}
        \end{subfigure}
        \caption{Bouncing $\gadget{2}$. Starting from the pattern $\gadget{2}$ (\subref{fig:bouncing2:initial}) we can either lollipop using the orange edge (\subref{fig:bouncing2:red}), or the blue one (\subref{fig:bouncing2:green}). Either way, the path's end moves to the right.}
        \label{fig:bouncing2}
    \end{figure}
    
    \begin{figure}
        \centering
        \begin{subfigure}{.3\textwidth}
            \centering
            \includegraphics[page=28,scale=0.4]{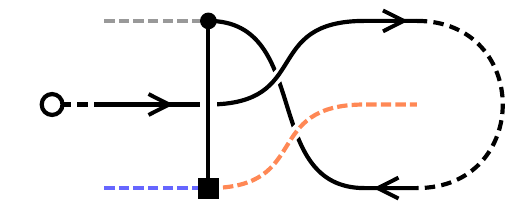}
            \caption{Pattern $\gadget{4}$}
            \label{fig:bouncing4:initial}
        \end{subfigure}
        \begin{subfigure}{.3\textwidth}
            \centering
            \includegraphics[page=29,scale=0.4]{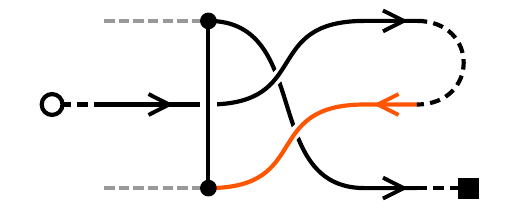}
            \caption{Pattern $\gadget{U''}$}
            \label{fig:bouncing4:red}
        \end{subfigure}
        \begin{subfigure}{.3\textwidth}
            \centering
            \includegraphics[page=30,scale=0.4]{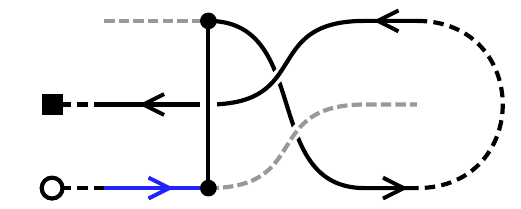}
            \caption{Uniquely determined pattern}
            \label{fig:bouncing4:green}
        \end{subfigure}
        \caption{Conducting $\gadget{4}$. Starting from the pattern $\gadget{4}$ (\subref{fig:bouncing4:initial}) we can either lollipop using the orange edge (\subref{fig:bouncing4:red}), or the blue one (\subref{fig:bouncing4:green}). In the first case, the path's end moves to the right, in the second case -- to the left.}
        \label{fig:bouncing4}
    \end{figure}
    
    \newpage
    \begin{lem}[Filling Lemma] \label{lem:fill}
        Consider a Hamiltonian path in $G_n$ starting at $\Lambda$ with endvertex in one of the gadgets in a bouncing pattern, i.e. $\gadget1$ or $\gadget2$. Then the two rightmost paths (arising from the two ways we can start lollipopping the path) in case of pattern \(\gadget1\) correspond to the words
        $w \hspace{2pt}
        \gadget{PQU}\hspace{2pt}
        \gadget{WSQU}\hspace{2pt}
        \gadget{WSQU}\hspace{2pt}
        \gadget{WSQU}
        ...$ and
        $w\hspace{2pt}
        \gadget{WRX}\hspace{2pt}
        \gadget{WSQU}\hspace{2pt}
        \gadget{WSQU}\hspace{2pt}
        \gadget{WSQU}
        ...$ where \(w\) denotes some common prefix. In the case of pattern $\gadget2$ the two paths are described by
        $w\hspace{2pt}
        \gadget{WSQU}\hspace{2pt}
        \gadget{PQU}\hspace{2pt}
        \gadget{PQU}\hspace{2pt}
        \gadget{PQU}
        ...$ and 
        $w\hspace{2pt}
        \gadget{WRX}\hspace{2pt}
        \gadget{PQU}\hspace{2pt}
        \gadget{PQU}\hspace{2pt}
        \gadget{PQU}
        ...$, where $w$ is some common prefix.
    \end{lem}
    \begin{proof}
        We begin with a bouncing pattern. Observe, that all subsequent patterns that end a path we can obtain by lollipopping before reaching the pac, are conducting ones.
        These number patterns will eventually repeat as there are only finitely many of them. Therefore,
        the resulting word describing the rightmost path has a periodic suffix. To obtain it, it is enough to analyse the behaviour after a bounded number of lollipops.
        
        First, we observe that we have the following possible transitions when lollipopping a path ending in a bouncing pattern:
        $\omega\gadget{1} \mapsto \{\omega\gadget{P3}, \omega\gadget{WR5}\}$\; and\;
        $\omega\gadget{2} \mapsto \{\omega\gadget{R6}, \omega\gadget{SQ7}\}$.
        For the conducting patterns we consider only the lollipops that move the path's end to the right:
        $\omega\gadget{3} \mapsto \omega\gadget{Q4}$,\; 
        $\omega\gadget{4} \mapsto \omega\gadget{UWS3}$,\; 
        $\omega\gadget{5} \mapsto \omega\gadget{XWS3}$,\; 
        $\omega\gadget{6} \mapsto \omega\gadget{XPQ7}$,\; and\;
        $\omega\gadget{7} \mapsto \omega\gadget{UPQ7}$. Here $\omega$ denotes some common prefix.
        
        Combining these transitions we obtain the possible traces for $\omega\gadget{1}$ and $\omega\gadget{2}$:
        \begin{align*}
        \omega\gadget{1}
        &\mapsto \omega \gadget{P3}
         \mapsto \omega \gadget{PQ4}
         \mapsto \omega \gadget{PQUWS3}
         \mapsto \dots \\
        \omega\gadget{1}
        &\mapsto \omega \gadget{WR5}
         \mapsto \omega \gadget{WRXWS3}
         \mapsto \omega \gadget{WRXWSQ4}
         \mapsto \omega \gadget{WRXWSQUWS3}
         \mapsto \dots \\
        \omega\gadget{2}
        &\mapsto \omega \gadget{R6}
         \mapsto \omega \gadget{RXPQ7}
         \mapsto \omega \gadget{RXPQUPQ7}
         \mapsto \dots \\
        \omega\gadget{2} 
        &\mapsto \omega \gadget{SQ7}
         \mapsto \omega \gadget{SQUPQ7}
         \mapsto \dots
        \end{align*}
        For an illustration, we present the first case in \cref{fig:fill}.
        
        Observe that the only letter pattern that matches the pattern \(\gadget{2}\) on the left (i.e. directly preceding it) is~\(\gadget{W'}\).
        Thus, in this case the last symbol of $\omega$ must be $\gadget{W}$, which gives us the description from the statement of the lemma.
        
        \begin{figure}[h]
            \centering
            \begin{subfigure}{\textwidth}
                \centering
                \includegraphics[page=33,scale=0.3]{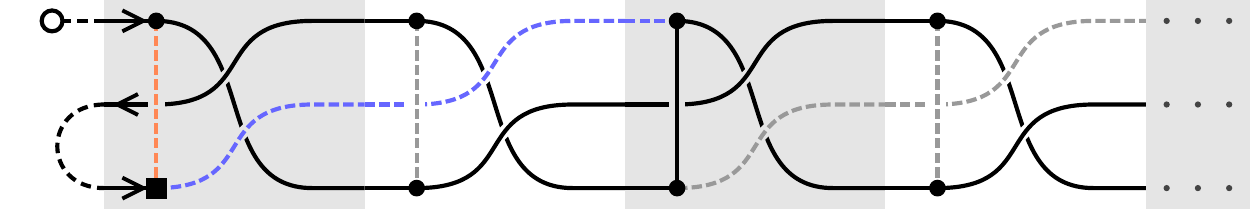}
                \caption*{\podpisik{1}\podpisik{}\podpisik{}\podpisik{}\\[10pt]}
            \end{subfigure}
            \begin{subfigure}{\textwidth}
                \centering
                \includegraphics[page=34,scale=0.3]{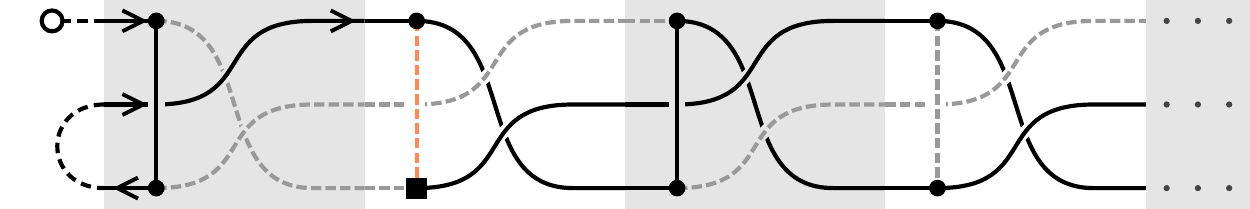}
                \caption*{\podpisik{P}\podpisik{3}\podpisik{}\podpisik{}\\[10pt]}
            \end{subfigure}
            \begin{subfigure}{\textwidth}
                \centering
                \includegraphics[page=35,scale=0.3]{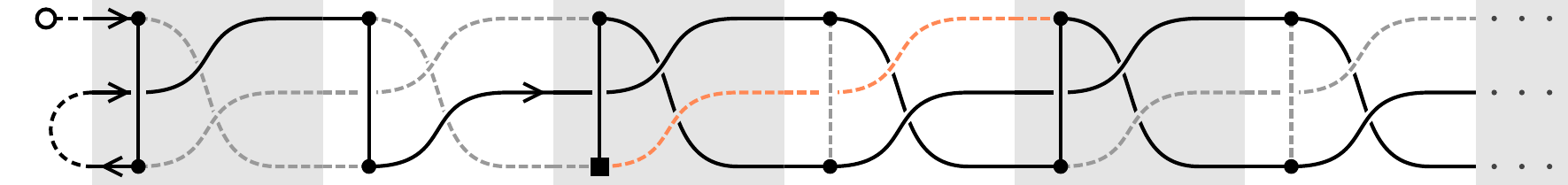}
                \caption*{\podpisik{P}\podpisik{Q}\podpisik{4}\podpisik{}\podpisik{}\podpisik{}\\[10pt]}
            \end{subfigure}
            \begin{subfigure}{\textwidth}
                \centering
                \includegraphics[page=36,scale=0.3]{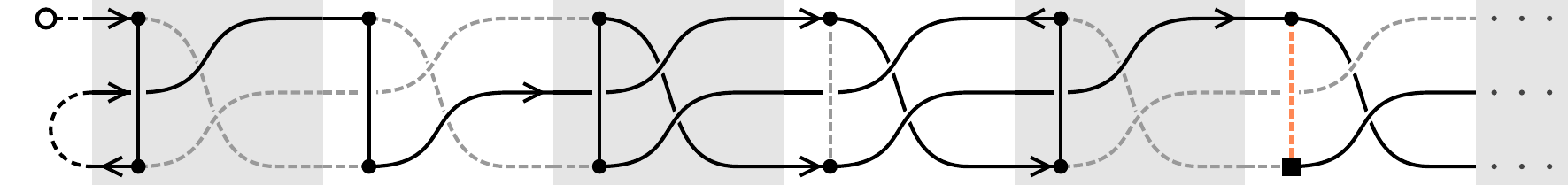}
                \caption*{\podpisik{P}\podpisik{Q}\podpisik{U''}\podpisik{W''}\podpisik{S}\podpisik{3}\\[2pt]}
            \end{subfigure}
            \caption{Subsequent Thomason's steps starting from the path ending in pattern $\gadget{1}$, using edges marked orange for lollipopping.
            }
            \label{fig:fill}
        \end{figure}
        \begin{figure}[b]
            \centering
            \begin{subfigure}{0.45\textwidth}
                \centering
                \includegraphics[page=51,scale=0.3]{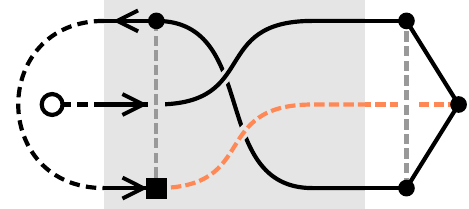}
                \caption*{\podpisik{5}\\[10pt]}
            \end{subfigure}
            \begin{subfigure}{0.45\textwidth}
                \centering
                \includegraphics[page=52,scale=0.3]{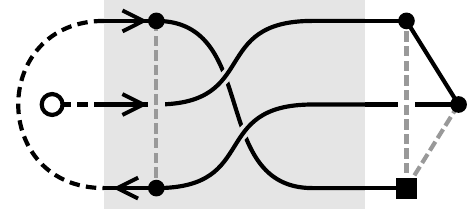}
                \caption*{\podpisik{}\podpisik{X}\podpisikpac\\[10pt]}
            \end{subfigure}
            \begin{subfigure}{0.45\textwidth}
                \centering
                \includegraphics[page=53,scale=0.3]{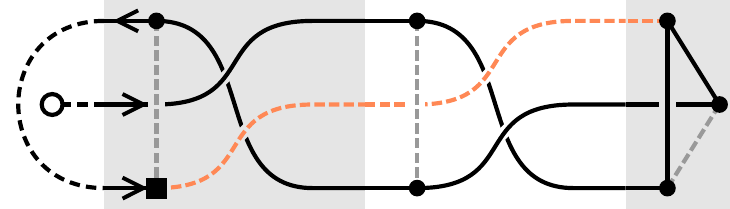}
                \caption*{\podpisik{5}\podpisik{}\\[10pt]}
            \end{subfigure}
            \begin{subfigure}{0.45\textwidth}
                \centering
                \includegraphics[page=54,scale=0.3]{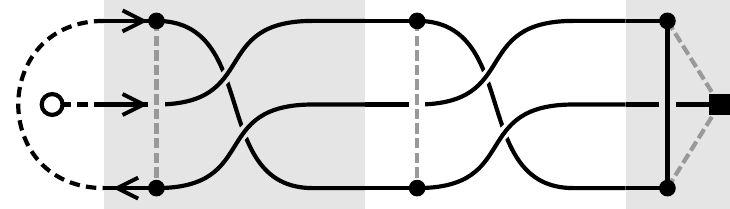}
                \caption*{\podpisik{}\podpisik{X}\podpisik{W}\podpisikpac\\[10pt]}
            \end{subfigure}
            \begin{subfigure}{0.45\textwidth}
                \centering
                \includegraphics[page=55,scale=0.3]{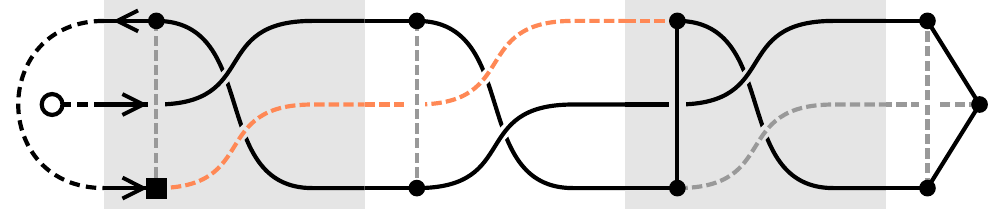}
                \caption*{\podpisik{5}\podpisik{}\podpisik{}\\[10pt]}
            \end{subfigure}
            \begin{subfigure}{0.45\textwidth}
                \centering
                \includegraphics[page=56,scale=0.3]{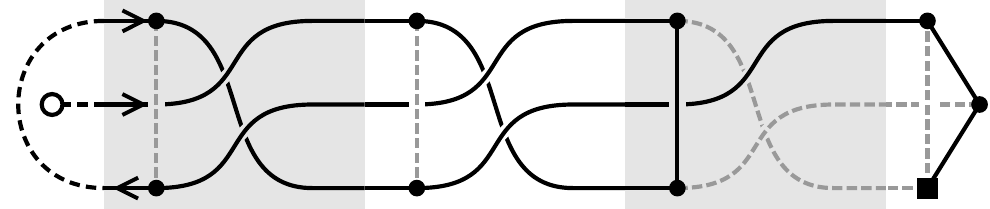}
                \caption*{\podpisik{}\podpisik{X}\podpisik{W}\podpisik{S}\podpisikpac\\[10pt]}
            \end{subfigure}
            \begin{subfigure}{0.45\textwidth}
                \centering
                \includegraphics[page=57,scale=0.3]{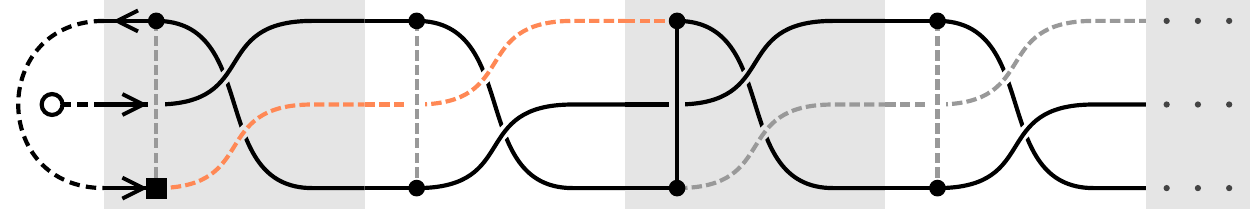}
                \caption*{\podpisik{5}\podpisik{}\podpisik{}\podpisik{}\\[2pt]}
            \end{subfigure}
            \begin{subfigure}{0.45\textwidth}
                \centering
                \includegraphics[page=58,scale=0.3]{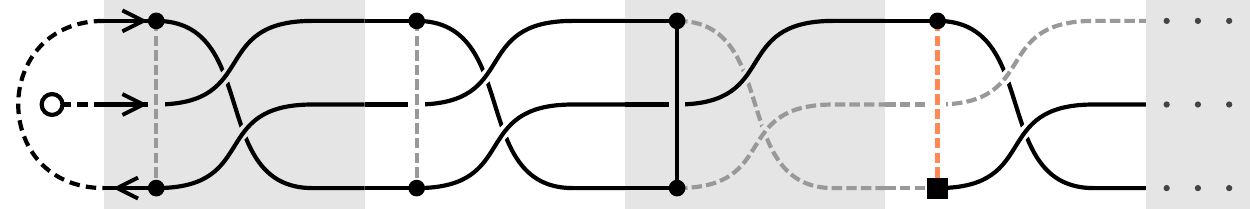}
                \caption*{\podpisik{X}\podpisik{W}\podpisik{S}\podpisik{3}\\[2pt]}
            \end{subfigure}
            \caption{All possible settings in which we may encounter the number pattern \(\gadget{5}\), together with the corresponding results of lollipopping that moves the path's endvertex to the right.
            \\[10pt]  
            }
            \label{fig:filling5}
        \end{figure}
        
        Finally, we also need to verify that when the end of the path eventually reaches the pac, the resulting description conforms to the repetitive formula -- that is, it is a prefix of the expected string. We do that by noting that
        $\omega \gadget{1} \mapsto \{
          \omega \gadget{P\$},
          \omega \gadget{W\$},
          \omega \gadget{WR\$}
        \}$,\;
        $\omega \gadget{2} \mapsto \{
          \omega \gadget{R\$},
          \omega \gadget{S\$},
          \omega \gadget{SQ\$}
        \}$,\;
        $\omega \gadget{3} \mapsto \{
          \omega \gadget{Q\$}
        \}$,\;
        $\omega \gadget{4} \mapsto \{
          \omega \gadget{U\$},
          \omega \gadget{UW\$},
          \omega \gadget{UWS\$}
        \}$,\;
        $\omega \gadget{5} \mapsto \{
          \omega \gadget{X\$},
          \omega \gadget{XW\$},
          \omega \gadget{XWS\$}
        \}$,\;
        $\omega \gadget{6} \mapsto \{
          \omega \gadget{X\$},
          \omega \gadget{XP\$},
          \omega \gadget{XPQ\$}
        \}$,\; and \;
        $\omega \gadget{7} \mapsto \{
          \omega \gadget{U\$},
          \omega \gadget{UP\$},
          \omega \gadget{UPQ\$}
        \}$\;
        are also valid lollipops, where $\gadget{\$}$ denotes that the encoded path has an endvertex in the pac. For illustration we present all transitions from $\omega\gadget{5}$ in \cref{fig:filling5}.
    \qedhere
    
    \end{proof}

    We are now prepared to prove the \cref{thm:main}.
    
    \begin{proof}[Proof of \cref{thm:main}]
    
    Consider a rightmost Hamiltonian path in the graph $G$. Each gadget is thus assigned a letter (see \cref{fig:letter_states}), and the path is assigned a word over the alphabet
    $\Sigma = \{\gadget{P}, \gadget{Q}, \gadget{U}, \gadget{W}, \gadget{R}, \gadget{X}, \gadget{S}\}$.
    Obviously, not all words in $\Sigma^{n}$ constitute a valid path, and to understand which do, we analyse the way each pattern behaves on its left and right edge cut (and which endpoints need to be connected to which). This analysis is compactly presented by introducing an automaton in \cref{fig:automaton_zwykly}. For the definitions and conventions pertaining to finite automata and regular languages used here see \cite{automata}.
    
    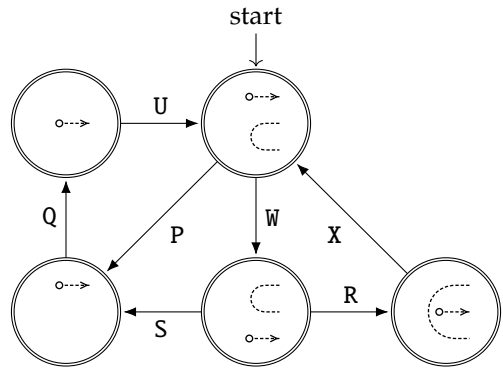
\begin{figure}[h]
        \centering
        \begin{tikzpicture}[arrows={-Latex},shorten >=1pt,node distance=2.5cm,on grid,auto]
            
            \node[state, accepting] (uv) {\includegraphics[page=47,scale=0.25]{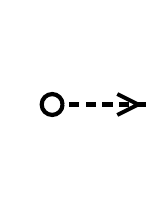}};
            \node[state, initial, initial where = above, accepting] (aa) [right = of uv] {\includegraphics[page=48,scale=0.25]{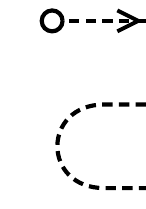}};
            
            \node[state, accepting] (q) [below = of uv] {\includegraphics[page=42,scale=0.25]{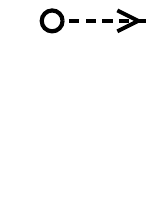}};
            \node[state, accepting] (r) [right = of q] {\includegraphics[page=43,scale=0.25]{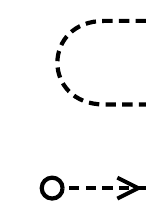}};
            \node[state, accepting] (x) [right = of r] {\includegraphics[page=44,scale=0.25]{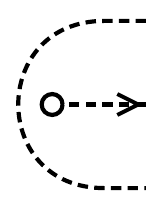}};
            
            \path (r) edge node {$\gadget{S}$} (q);
            \path (uv) edge node [above] {$\gadget{U}$} (aa);
            \path (aa) edge node {$\gadget{W}$} (r);
            \path (aa) edge node {$\gadget{P}$} (q);
            \path (q) edge node {$\gadget{Q}$} (uv);
            \path (r) edge node {$\gadget{R}$} (x);
            \path (x) edge node {$\gadget{X}$} (aa);
            
        \end{tikzpicture}
        \caption{Automaton accepting exactly descriptions of rightmost Hamiltonian paths for any value of $n$.
        Figures in nodes represent the way the path traverses the graph up to the cut between considered gadgets.
        }
        \label{fig:automaton_zwykly}
    \end{figure}
    
    Let's call the language of this automaton $ {\mathcal{J}} $.
    Since this language includes words of arbitrary lengths, and we are only interested in words of length $n$ exactly, let $\mathcal{J}_{n} = {\mathcal{J}} \cap \Sigma ^{n}$.
    Taking a closer look at the automaton, we may notice that
    \begin{equation*}
    \mathcal{J}_{n} = \Prefx_{n}( 
    \{
        \gadget{WRX},
        \gadget{PQU},
        \gadget{WSQU}
    \}
    ^{*})
    \text{,}
    \end{equation*} 
    where $\Prefx_n(L)$ is a language consisting of all prefixes of length $n$ exactly of words in the language $L$.
    
    Let us now consider a smaller alphabet:
    $\Gamma = \left \{ \base{A}, \base{T}, \base{G}\right \}$
    with the map 
    $ \gadget{PQU} \mapsto \base{A} $, $ \gadget{WRX} \mapsto \base{T} $, and $ \gadget{WSQU} \mapsto \base{G} $.
    We also define a partial order $<$ on \(\Gamma^{*}\)  inductively as $ \base{A} < \base{T} < \base{G}$ and 
    \[
        \beta u < \gamma w  \iff    \begin{cases}
                                        \beta < \gamma  & \text{ if } \beta \neq \gamma\text,\\
                                        u < w           & \text{ if } \beta = \gamma \neq \base{T}\text,\\
                                        w < u           & \text{ if } \beta = \gamma = \base{T}\text,
                                    \end{cases}
    \]
    where $\beta, \gamma \in \Gamma $ and $u, w \in \Gamma^{*}$.
    
    It may very well happen that our path fails to be divided evenly by the words $\gadget{WRX}$, $\gadget{PQU}$, $\gadget{WSQU}$. More precisely, the path can be seen as a concatenation of these words (i.e. $\{
        \gadget{WRX},
        \gadget{PQU},
        \gadget{WSQU}
    \}^{*}$), concatenated once more with a prefix of one of these words. This is a consequence of the definition of $\mathcal{J}_{n}$. We need to handle the case when the path ends with only a proper prefix of one of these words. If this remaining suffix of the path contains two or more gadgets, we already know which of $\base{A}$, $\base{T}$, $\base{G}$ the last symbols must correspond to, as the underlying words do not share more than one initial symbol.
    One-letter unmatched suffixes would make such mapping ambiguous. To remedy this, we introduce yet another symbol -- $\base{C}$ to describe suffix $\gadget{W}$, while keeping $\base{A}$ to describe suffix \gadget{P}. Now we need to extend the order with $\base{A} < \base{C}$, what makes it consistent and exhaustive as we never need to compare $\base{T}$ or $\base{G}$ with $\base{C}$. Therefore any rightmost path can be described by a word in the language \(\mathcal{K}\) defined as follows.
    \[
     \mathcal{K} = \{\base{A},\base{T},\base{G}\}^{*}\{\base{C},\varepsilon\}
    \]

    Addition of $\base{C}$ makes the definition of the map quite verbose, so we include the following formal, inductive formula of $\varphi \colon \mathcal{J} \rightarrow \mathcal{K}$:
    \[
        \varphi(\lambda) = \begin{cases}
        \base{A} &\text{if }  \lambda  \in \{\gadget{P}, \gadget{PQ}, \gadget{PQU}\}\text,
        \\
        \base{T} &\text{if }  \lambda  \in \{\gadget{WR}, \gadget{WRX}\}\text,
        \\
        \base{G} &\text{if } \lambda  \in \{\gadget{WS}, \gadget{WSQ}, \gadget{WSQU} \}\text,
        \\
        \base{C} &\text{if } \lambda \in \{\gadget{W}\}\text,
        \\
        \base{A}\varphi(\omega) & \text{if } \lambda = \gadget{PQU} \omega
        \text,\; \omega \neq \varepsilon\text,
        \\
        \base{T} \varphi(\omega) & \text{if } \lambda = \gadget{WRX} \omega  
        \text,\; \omega \neq \varepsilon\text,
        \\ 
        \base{G} \varphi(\omega) & \text{if } \lambda = \gadget{WSQU} \omega  
        \text,\; \omega \neq \varepsilon\text.
        \end{cases}
    \]
    Finally, let
    \[
        \mathcal{L}_{n} = \varphi (\mathcal{J}_{n}) \text{.}
    \]
    
    Observe that $\mathcal{L}_n \subseteq \mathcal{K}$, and while order $<$ is not linear in $\mathcal{K}$, it is linear in $\mathcal{L}_{n}$.
    
    By the \nameref{lem:init}, the first rightmost path upon starting Thomason's algorithm on $C_0$ corresponds to the word $\base{AAA}...\base{A}$ -- the least word in $\mathcal{L}_{n}$ with respect to the order $<$, and the last rightmost path before we get to $C_1$ corresponds to the word $\base{GGG}...\base{G}$ (or possibly $\base{GGG}...\base{GC}$ if \(n\) is congruent to \(1\) modulo \(4\)) -- the greatest word in $\mathcal{L}_{n}$.
    
    Let us consider all rightmost paths that arise during the algorithm \(p_1, p_2, \dots, p_t\), in the order of appearance.
    By the observation above, \(p_1\) and \(p_t\) are encoded by the least and the greatest words in \(\mathcal{L}_n\) respectively.
    Consider now \(p_i\) and \(p_{i+1}\), where \(i \in \{1,2,\dots, t - 1\}\).
    Suppose that both paths are encoded by the same word in $\mathcal{L}_n$. Recall that there are two distinct ways to traverse the pac, and the corresponding paths, encoded by the same word, are always neighbours in \(\koppa(G_n)\). This implies that $p_{i+1}$ is an immediate successor of $p_i$ during the Thomason's algorithm.
    Otherwise $p_i$ and $p_{i+1}$ are encoded by different words, and all paths that appear during algorithm between $p_i$ and $p_{i+1}$ are ending in a number pattern. By the \nameref{lem:bounce}, there is a path between $p_i$ and $p_{i+1}$ that ends either in the pattern \(\gadget{1}\) or \(\gadget{2}\).
    
    Now the \nameref{lem:fill} implies that the paths $p_i$ and $p_{i+1}$ are of the form $w \gadget{PQU}\left( \gadget{WSQU}\right)^*$ and $ w \gadget{WRX}\left( \gadget{WSQU}\right)^*$ in the case of pattern $\gadget{1}$, or $ w  \gadget{WSQU} \left(\gadget{PQU}\right)^*$ and $ w \gadget{WRX}\left( \gadget{PQU}\right)^*$ otherwise, where $w$ is some common prefix.
    Clearly the words in $\mathcal{L}_{n}$ corresponding to these paths form consecutive pairs with respect to the order $<$, and so while running the algorithm we move either to an immediate successor of the path $p_i$ with respect to the order induced via $\varphi$ from $<$.

    This proves that while running Thomason's algorithm we visit all words in $\mathcal{L}_{n}$ exactly in order $<$.
    Observe that we make at most $2n$ lollipops between any two distinct rightmost paths.
    Hence, to count the number of steps of the algorithm, up to a factor linear in $n$, it suffices to count the number of words in $\mathcal{L}_{n}$.
    
    Let $a_k$ be the number of words in the language $\mathcal{L}_{k}$. One can easily verify the following recurrence
    \begin{align*}
        a_0 &= 1\text,\\
        a_1 &= 2\text,\\
        a_2 &= 3\text,\\
        a_3 &= 3\text,\\
        a_k &= 2 a_{k-3} + a_{k-4}  \hspace{20pt} \text{  for  }k \geq 4\text{.}
    \end{align*}
    
    We use method of Generating Functions to derive the asymptotic behaviour of this sequence, see e.g. \cite{analytic} for an introduction of this method.
    Consider $\mathcal{A}(z) = \sum_{k=0}^{\infty} a_k z^k$. We get the functional equation
    \[
        \mathcal{A}(z) = 1 + 2 z + 3 z^2 + z^3 + 2 z^3 A(z) + z^4 A(z)\text{,}
    \]
    which one can easily solve to
    \[
        \mathcal{A}(z) = \frac{1 + 2 z + 3 z^2 + z^3}{1 - 2 z^3 - z^4}\text{.}
    \]
    
    Thus, we get that asymptotically there are $\Theta(c^n)$ such words, where $\frac{1}{c}$ is the least modulus among the roots of $z^4 + 2 z^3 -1$ (which equals approximately $1.3953...$).
    
    To conclude the proof, observe that the number of lollipops between two rightmost paths amortises to a constant.
    Each time we change the path meaningfully (i.e., the image under $\varphi$ changes) by $k$ letters, we need to introduce $2$ changes by $k - 1$ letters before another change by $k$ letters.
    This is a geometric pattern
    (akin to incrementing a binary counter),
    so the total number of lollipops amortises to the number of visited rightmost paths times a constant. \qedhere
    
    \end{proof}
    
    \section*{Concluding remarks}
    
    There are alternative, and arguably simpler, ways to prove that Thomason's algorithm takes exponential time on $G_n$, however we believe that our proof provides valuable insight into the algorithm's behaviour.
    In particular, it might be helpful in development of new, faster algorithms that find a second Hamiltonian cycle in a given cubic graph.
    
    It would be interesting to find a family of cubic graphs with even faster growth of the number of steps taken by Thomason's algorithm.
    
    It is also worth noticing that Eppstein's algorithm \cite{eppstein} can, in the particular case of \(G_n\) and Krawczyk's graphs, solve the problem of finding \emph{all} Hamiltonian cycles efficiently despite being exponential in the worst case. This follows form the fact that both families are extremely constrained -- branching a bounded number of times already forces a single cycle, which Eppstein's deduction rules can later find in linear time.
    
    \renewcommand{\abstractname}{Acknowledgements}
    \begin{abstract}
        Words can hardly express our gratitude to our friends at Beit for their support, criticism, and insightful ideas.
    \end{abstract}

\bibliographystyle{alpha}
\addcontentsline{toc}{section}{References}
\bibliography{references}

\end{document}